\documentclass[english,letterpaper,11pt,reqno]{amsart}
\usepackage{appendix}
\usepackage{amsbsy}
\usepackage{scalerel}
\usepackage{tcolorbox}
\usepackage{amsfonts}
\usepackage{amsmath}
\usepackage{amssymb}
\usepackage{amsthm}
\usepackage{graphicx}
\usepackage{ifthen}
\usepackage{textcomp}
\usepackage{enumitem, color, amssymb}
\usepackage{dsfont}
\usepackage{mathrsfs}
\usepackage{calrsfs}
\usepackage[bookmarksnumbered,colorlinks]{hyperref}
\hypersetup{colorlinks=true, linkcolor=blue, citecolor=blue}
\usepackage{cancel}
\usepackage{setspace}

\newcommand{\lra}{\longrightarrow}

\newcommand{\RR}{\mathbb{R}}

\newcommand{\vepo}{\epsilon_{\scalebox{0.4}{\emph{L}}}}

\makeatletter
\newcommand*{\defeq}{\mathrel{\rlap{%
                     \raisebox{0.25ex}{$\m@th\cdot$}}%
                     \raisebox{-0.25ex}{$\m@th\cdot$}}%
                     =}
\makeatother

\makeatletter
\newcommand*\owedge{\mathpalette\@owedge\relax}
\newcommand*\@owedge[1]{%
  \mathbin{%
    \ooalign{%
      $#1\m@th\bigcirc$\cr
      \hidewidth$#1\m@th\wedge$\hidewidth\cr
    }%
  }%
}
\makeatother

\newtheorem{thm}{Theorem}
\newtheorem{lemma}{Lemma}

\newtheorem{defn}{Definition}
\newtheorem{prop}{Proposition}
\newtheorem*{definition-non}{Definition}
\newtheorem*{theorem-non}{Theorem}
\newtheorem*{proposition-non}{Proposition}
\newtheorem*{lemma-non}{Lemma}
\newtheorem*{corollary-non}{Corollary}

\newcommand{\beqa}{\begin{eqnarray}}
\newcommand{\beq}{\begin{equation}}
\newcommand{\eeqa}{\end{eqnarray}}
\newcommand{\eeq}{\end{equation}}

\newcommand\ipl[2]{\langle {#1},{#2}\rangle_{\!g_{\scalebox{0.3}{\emph{L}}}}}
\newcommand\ipr[2]{\langle {#1},{#2}\rangle_{\!\scalebox{0.7}{\emph{g}}}}
\newcommand\iph[2]{\langle {#1},{#2}\rangle_{\!\scalebox{0.7}{\emph{h}}}}

\newcommand\ww[2]{#1 \wedge #2}
\newcommand\imp{\hspace{.2in}\Rightarrow\hspace{.2in}}
\newcommand\PL{P^{\perp_{\scalebox{0.35}{\emph{L}}}}}
\newcommand\Ph{P^{\perp_{\scalebox{0.5}{\emph{h}}}}}

\newcommand\cd[2]{\nabla_{\!#1}{#2}}

\newcommand\gL{g_{\scalebox{0.4}{\emph{L}}}}

\newcommand\Ric{\text{Ric}_{\scalebox{0.6}{\emph{g}}}}

\newcommand\comma{\hspace{.2in},\hspace{.2in}}
\newcommand\commas{\hspace{.1in},\hspace{.1in}}

\newcommand\hsl{*_{\scalebox{0.4}{\emph{L}}}}
\newcommand\hsh{*_{\scalebox{0.5}{\emph{h}}}}

\newcommand\hsr{*_{\scalebox{0.5}{\emph{g}}}}

\newcommand\col{\hat{R}_{\scalebox{0.4}{\emph{L}}}}
\newcommand\cowl{\hat{W}_{\scalebox{0.4}{\emph{L}}}}
\newcommand\cow{\hat{W}}
\newcommand\cx{\Lambda_{\scalebox{0.5}{\emph{$\mathbb{C}$}}}^2}
\newcommand\Tsec{\text{sec}_{\scalebox{0.4}{\emph{$L$-$\hat{R}$}}}}

\newcommand\Ssec{\text{sec}_{\scalebox{0.4}{\emph{$L$-$\hat{S}$}}}}
\newcommand\gsec{\text{sec}_{\scalebox{0.6}{\emph{g}}}}
\newcommand\co{\hat{R}_{\scalebox{0.6}{\emph{g}}}}
\newcommand\cW{\hat{W}_{\scalebox{0.6}{\emph{g}}}}

\newcommand\cooh{{\hat{R}^{\scalebox{0.5}{\emph{h}}}}_{\scalebox{0.4}{\emph{L}}}}

\newcommand\cS{\hat{S}_{\scalebox{0.6}{\emph{g}}}}
\newcommand\cA{\hat{A}_{\scalebox{0.6}{\emph{g}}}}

\newtheorem{innercustomgeneric}{\customgenericname} 
\providecommand{\customgenericname}{}
\newcommand{\newcustomtheorem}[2]{%
  \newenvironment{#1}[1]
  {%
   \renewcommand\customgenericname{#2}%
   \renewcommand\theinnercustomgeneric{##1}%
   \innercustomgeneric
  }
  {\endinnercustomgeneric}
}

\newcustomtheorem{customthm}{Theorem}
\newcustomtheorem{customlemma}{Lemma}

\begin{document}
\title[]{On the Petrov Type of a 4-manifold}
\author[]{Amir Babak Aazami}
\address{Clark University\hfill\break\indent
Worcester, MA 01610}
\email{aaazami@clarku.edu}

\maketitle
\begin{abstract}
On an oriented 4-manifold, we examine the geometry that arises when the curvature operator of a Riemannian or Lorentzian  metric $g$ commutes, not with its own Hodge star operator, but rather with that of another semi-Riemannian metric $h$ that is a suitable deformation of $g$.  We classify the case when one of these metrics is Riemannian and the other Lorentzian by generalizing the concept of Petrov Type from general relativity; the case when $h$ is split-signature is also examined. The ``generalized Petrov Types" so obtained are shown to relate to the critical points of $g$'s sectional curvature, and sometimes yield unique normal forms.  They also carry topological information independent of the Hitchin-Thorpe inequality, and yield a direct geometric formulation of ``almost-Einsten" metric via the Ricci or sectional curvature of $g$.
\end{abstract}

\section{Introduction}
The Petrov classification of 4-dimensional spacetimes, due to A. Petrov \cite{petrov}, is a foundational result in gravitational physics, a classification of Lorentzian metrics by their Weyl curvature tensors; see \cite{o1995, stephani2009} for modern treatments.  Somewhat less well known is its reformulation due to J. Thorpe \cite{thorpe}, who showed that a spacetime's ``Petrov Type" is in fact the Lorentzian analogue of the ``normal form" of an oriented 4-dimensional Riemannian manifold.  The latter refers to those Riemannian 4-manifolds $(M,g)$ whose curvature operator $\co\colon \Lambda^2\lra \Lambda^2$ is determined by knowledge of just the critical points and values of its sectional curvature $\gsec$, when the latter is viewed as the quadratic form of $\co$.  The most well known examples of such 4-manifolds, discovered independently by M. Berger \cite{berger} and I. Singer \& Thorpe \cite{thorpe2}, are Einstein metrics; i.e., those Riemannian metrics $g$ whose Ricci tensor $\Ric$ satisfies $\Ric  = \lambda g$ for some $\lambda \in \RR$.  The reason for this is that in dimension 4, the Einstein condition arises precisely when $g$'s Hodge star operator $*\colon \Lambda^2 \lra \Lambda^2$ commutes with $\co$:
\beqa
\label{eqn:comm1}
\co \circ * = * \circ \co.
\eeqa
As shown in \cite{thorpe}, this commutativity is also the origin of Petrov Type, but with the Weyl curvature in place of $\co$.  Given this common origin, the purpose of this article is to generalize \eqref{eqn:comm1} by examining the geometry that arises when $\co$ and $*$ arise from \emph{different metrics} on $M$; i.e., when the curvature operator $\co$ of a Riemannian or Lorentzian metric commutes, not with its own Hodge star $*$, but rather with that of another semi-Riemannian metric $h$ that is a suitable deformation of $g$.  Semi-Riemannian here means that $h$ can be a Riemannian, a Lorentzian, or a ``split-signature" metric; i.e., a metric of signature $(2,2)$.  As we will show, in each of these cases there is geometric and sometimes even topological information contained in the symmetric or ``Weyl-like part" of $\co$, that is to say, in that portion of $\co$ that commutes with $h$'s Hodge star operator $\hsh$:
\beqa
\label{eqn:Weylpart}
\cS \defeq \frac{1}{2}(\co \pm \hsh \circ \co \circ \hsh) \imp \cS \circ \hsh = \hsh \circ \cS.
\eeqa
(If $\hsh$ were replaced by $*$ here, then $\cS = \cW + \frac{\text{scal}_{\scalebox{0.5}{\emph{g}}}}{12}I$, where $\cW$ is the Weyl curvature operator of $g$ and $\text{scal}_{\scalebox{0.5}{\emph{g}}}$ is its scalar curvature.) What we mean by ``suitable deformation of $g$" is simply that $h$ has orthonormal frames in common with $g$. For such $h$, we prove the following:
\begin{enumerate}[leftmargin=*]
\item[i.] When $g$ is Riemannian and $h$ Lorentzian, we take the latter to be of the form $h \defeq g - 2T^{\flat} \otimes T^\flat$ with $T$ a $g$-unit length vector field on $M$.  For such $h$, we then proceed to generalize the notion of Petrov Type by classifying the $\cS$'s in \eqref{eqn:Weylpart} analogously to \cite{thorpe}.  As shown in Theorems \ref{thm:JordanP} and \ref{thm:SS}, there will be three Petrov Types, given by the Jordan normal forms of $\cS$ when the latter is viewed as a complex-linear map via the complex multiplication on $\Lambda^2$ provided by $\hsh^2 = -1$.
\item[ii.] The best case is when $\cS = \co$ (i.e., when $\co$ itself commutes with $\hsh$).  It yields a notion of ``almost-Einstein" metric generalizing $\Ric = \lambda g$, whether $h$ is Lorentzian or split-signature.  While the latter case is more direct (Theorem \ref{thm:sech}), the former is also noteworthy (Theorems \ref{prop:psE} and \ref{thm:secsec}), because ``almost-Einstein" is the best one can do if $M$ is compact: There are no non-flat Einstein metrics when $M$ has Euler characteristic zero \cite{berger}, as it must if it is to admit a Lorentzian metric.  Even more, Theorems \ref{prop:psE} and \ref{thm:n} show that when $h$ is Lorentzian, $g$'s Petrov Type is related to the \emph{number} of critical points of $\gsec$, and that there is a further topological obstruction: The signature of $M$ must be zero, too.  (This is derived independently of the Hitchin-Thorpe inequality \cite{HT-thorpe,HT-hitchin}; i.e., without assuming the existence of an Einstein metric on $M$.)
\item[iii.] Finally, we examine the case when $g$ is Lorentzian and $h \defeq g + 2T^{\flat} \otimes T^\flat$ is Riemannian, or when both are Riemannian of any form. These yield the best results, provided one uses a variant of $\co$ that is self-adjoint with respect to $h$.  With this modification, Theorem \ref{thm:last} shows that there is a \emph{unique} normal form (Petrov Type), that it is determined entirely by the critical points and values of a certain quadratic form, and that it yields an ``almost-Einstein" as in ii. above.
\end{enumerate} 

\section{The Hodge star, normal forms, and Petrov Types}
\label{sec:Intro}
In this section we provide a brief overview as well as some historical remarks regarding normal forms and the Hodge star operator in both the Riemannian and Lorentzian settings.  The study of ``normal forms" for curvature operators is motivated by a well known fact from linear algebra: On any finite-dimensional inner product space $(V,\langle\,,\rangle)$, a self-adjoint linear transformation $T\colon V\lra V$ is determined by the critical points and values of its \emph{quadratic form} function, defined on the unit sphere of $V$ as follows:
\beqa
\label{eqn:la}
v \mapsto \langle Tv,v\rangle \comma |v| = 1.
\eeqa
The connection to Riemannian geometry is that the sectional curvature of a Riemannian manifold $(M,g)$ can also be realized as a quadratic form.  First, fix $p \in M$ and observe that the second exterior product $\Lambda^2 \defeq \Lambda^2(T_pM)$ inherits an inner product $\ipr{\,}{}$ from $g$ as follows: For all $v_1,v_2,w_1,w_2 \in T_pM$,
\beqa
\label{eqn:ip}
\ipr{\ww{v_1}{w_1}}{\ww{v_2}{w_2}}  \defeq \text{det}\begin{bmatrix}
g(v_1,v_2) & g(v_1,w_2)\\
g(w_1,v_2) & g(w_1,w_2)
\end{bmatrix}\cdot
\eeqa
With this in hand, we may express the Riemann curvature 4-tensor $R$ as a linear map $\co\colon \Lambda^2 \lra \Lambda^2$ called the \emph{curvature operator}, whose action $v \wedge w \mapsto \co(v\wedge w)$ is defined to be the unique 2-vector satisfying
\beqa
\label{eqn:minus0}
\ipr{\co(v\wedge w)}{x\wedge y} \defeq -R(v,w,x,y)\hspace{.2in}\text{for all $x,y \in T_pM$}.
\eeqa
Owing to the symmetry $R_{ijkl} = R_{klij}$, observe that $\co$ is $\ipr{\,}{}$-self-adjoint. It is with respect to $\co$ that the sectional curvature of $g$ is a quadratic form.  Indeed, for any orthonormal pair $v,w \in T_pM$, the sectional curvature $\gsec$ of the 2-plane $v\wedge w$ is
\beqa
\label{sec:v0}
\gsec(v\wedge w) \defeq R(v,w,w,v) = \underbrace{\,\ipr{\co(v\wedge w)}{v\wedge w}\,}_{\text{``$\langle Tv,v\rangle$"}}.
\eeqa
(Our sign convention is $R(a,b,c,d) \defeq g(\cd{a}{\!\cd{b}{\,c}}-\cd{b}{\!\cd{a}{c}}-\cd{[a,b]}{c},d)$.) Given this analogy with \eqref{eqn:la}, it is thus natural to study $\co$ by studying the critical point behavior of $\gsec$\,---\,the goal being to classify those curvature tensors which are determined by knowledge of just the critical point structure of $\gsec$. Such metrics are then said to have a ``normal form."  (Observe that this problem is more difficult than \eqref{eqn:la} because not all unit 2-vectors in $\Lambda^2$ correspond to 2-planes in $T_pM$.  Rather, only those unit 2-vectors $\xi$ that are decomposable do; i.e., those that can be written as $\xi = \ww{v}{w}$ for $v,w \in T_pM$.)  As shown in \cite{thorpe2}, if one restricts attention to the class of Einstein metrics on  a smooth 4-manifold $M$, then such metrics are indeed determined solely by the critical point structure of $\gsec$. As mentioned in the Introduction, what makes this beautiful result possible is that, in dimension 4, Einstein metrics $g$ are exactly those that satisfy \eqref{eqn:comm1}, where the Hodge star operator $\hsr$ is defined as sending any 2-plane $v \wedge w$ to its $\ipr{\,}{}$-orthogonal complement.  The precise definition is
\beqa
\label{eqn:hs}
\ww{\xi}{\hsr\eta} \defeq \ipr{\xi}{\eta}\,dV \comma \xi,\eta \in \Lambda^2,
\eeqa
where $dV$ is the orientation form in the 1-dimensional space $\Lambda^4(T_pM)$.  (As we'll see below, these definitions would not change if we replaced the Riemannian metric $g$ with a Lorentzian metric $\gL$.)  Thanks to $\hsr$, there is in fact another way to express the Einstein condition in dimension 4  (see \cite{thorpe2}):
\beqa
\label{eqn:gsec_def}
\text{$g$ is Einstein $\iff \gsec(v\wedge w) = \gsec(\hsr(v\wedge w))$.}
\eeqa
I.e., $g$ is Einstein if and only if the sectional curvature of each 2-plane is equal to that of its orthogonal complement.  With these preliminaries out of the way, let us now comment briefly on what else has been done regarding normal forms in the Riemannian setting.  Continuing the study of the pointwise sectional curvature function $\gsec$, Thorpe analyzed its zero set when $\gsec \geq 0$ in \cite{thorpe4}, and also gave a condition for a 4-manifold to have $\gsec > 0$, in \cite{thorpe3}. Dimension 4 is particularly rich: Aside from Einstein metrics, the class of K\"ahler 4-manifolds with positive Ricci curvature also possesses a normal form, as shown in \cite{johnson}. In fact Johnson also found normal forms for 6-dimensional K\"ahler manifolds with positive sectional curvature, in \cite{johnson2}.  As one may imagine, however, in higher dimensions the behavior of $\gsec$ becomes more difficult to analyze in general, as shown in \cite{zoltek}; indeed, aside from the K\"ahler case in dimension 6, the author knows of only one other higher-dimensional result, in \cite{zoltek2}, who found normal forms for a special class of curvature operators in dimension 5.  Nevertheless, the study of sectional curvature related to and inspired by these works continues to the present day; e.g., on a variant of $\co$ known as the \emph{curvature operator of the second kind} (see, e.g., the recent works \cite{gursky,li}).  So, too, does the quest for normal forms\,---\,even in dimension 4.  Indeed, \cite[Prop.~2.4]{cao} recently showed that gradient Ricci 4-solitons, which satisfy
$$
\Ric + \text{Hess}_{\scalebox{0.5}{\emph{g}}} f = \lambda g
$$
for some smooth function $f$ on $M$ and $\lambda \in \RR$, also possess a normal form with respect to the operator $\co + \hat{H}_{\scalebox{0.6}{\emph{g}}}$, where $\hat{H}_{\scalebox{0.6}{\emph{g}}} \defeq \frac{1}{2}\text{Hess}_{\scalebox{0.5}{\emph{g}}} f {\footnotesize \owedge}\,g$.
\vskip 6pt
Now we turn to Lorentzian geometry.  First, recall that a \emph{Lorentzian metric} $\gL$ on a smooth manifold $M$ is a smooth nondegenerate metric with signature $(-\!+\!+\!\cdots+)$.  The absence of positive-definiteness implies that all nonzero vectors $X \in TM$ come in three flavors:
\beqa
\text{$X$ is}~\left\{\begin{array}{ccc}
\text{``spacelike"} &\text{if}& \text{$\gL(X,X) > 0$},\\
\text{``timelike"} &\text{if}& \text{$\gL(X,X) < 0$},\\
\text{``lightlike"} &\text{if}& \text{$\gL(X,X) = 0$}.
\end{array}\right.\nonumber
\eeqa
The Hodge star operator $\hsl$ of a Lorentzian 4-manifold $(M,\gL)$ is defined in perfect analogy with that of a Riemannian metric $g$ on $M$; however, owing to the different signature, let us analyze $\hsl$ with some care.  To begin with, in order to have orthonormal frames in common with $g$, our $\gL$'s will not be chosen arbitrarily, but rather as follows: For a suitable choice of $g$-unit length vector field $T$, let us form the metric
\beqa
\label{eqn:gL}
\gL \defeq g -2T^{\flat}\otimes T^{\flat},
\eeqa
where $T^{\flat} \defeq g(T,\cdot)$ is the one-form $g$-metrically equivalent to $T$.  Notice that $\gL(T,T) = -1$, so that $T$ is unit \emph{timelike} with respect to $\gL$.  \emph{Note also that any $g$-orthonormal basis containing $T$ is also a \emph{$\gL$}-orthonormal basis, and vice-versa}\,---\,a ``deformation of $g$" as we called it in the Introduction.  Generally speaking, the more ``distinguished" $T$ is\,---\,e.g., if it is closed, $dT^{\flat} = 0$, or if it is a Killing vector field, $\mathfrak{L}_Tg = 0$\,---\,the more similar the properties of $g$ and $\gL$ will be; see \cite{olea} for a careful treatment.  In any case, whatever the choice of $T$, if we take an  oriented local $\gL$-orthonormal frame $\{e_1,e_2,e_3,e_4\}$ with $e_1 \defeq T$, then the 2-vectors
\beqa
\label{eqn:Hbasis0}
\{\ww{e_1}{e_2}\,,\,\ww{e_1}{e_3}\,,\,\ww{e_1}{e_4}\,,\, \ww{e_3}{e_4}\,,\, \ww{e_4}{e_2}\,,\, \ww{e_2}{e_3}\}
\eeqa
will comprise an orthonormal basis for $\Lambda^2$ with respect to the Lorentzian inner product $\ipl{\,}{}$ on $\Lambda^2$, which is defined as $\ipr{\,}{}$ was in \eqref{eqn:ip}:
\beqa
\label{eqn:ipL}
\ipl{\ww{v_1}{w_1}}{\ww{v_2}{w_2}} \defeq \text{det}\begin{bmatrix}
\gL(v_1,v_2) & \gL(v_1,w_2)\\
\gL(w_1,v_2) & \gL(w_1,w_2)
\end{bmatrix}\cdot
\eeqa
(Note that the  first three basis elements in  \eqref{eqn:Hbasis0} are all timelike,
$$
\ipl{\ww{e_1}{e_i}}{\ww{e_1}{e_i}} = -1 \comma i=2,3,4,
$$
so that $\ipl{\,}{}$ has signature $(-\!-\!-\!+\!++)$.)  Now we can define the Hodge star operator $\hsl$ with respect to $\gL$, in perfect analogy with \eqref{eqn:hs}:
$$
\ww{\xi}{\hsl\eta} \defeq \ipl{\xi}{\eta}\,dV \comma \xi,\eta \in \Lambda^2.
$$
Bearing in mind that $\gL(e_1,e_1)  = -1$, observe that the action of $\hsl$ on the basis \eqref{eqn:Hbasis0} is
\beqa
\arraycolsep=1.4pt\def\arraystretch{1.5}
\left\{\begin{array}{lr}
\hsl(\ww{e_1}{e_2}) = -\ww{e_3}{e_4},\\
\hsl(\ww{e_1}{e_3}) = -\ww{e_4}{e_2},\\
\hsl(\ww{e_1}{e_4}) = -\ww{e_2}{e_3},
\end{array}\right.  \comma \arraycolsep=1.4pt\def\arraystretch{1.5}
\left\{\begin{array}{lr}
\hsl(\ww{e_3}{e_4}) = \ww{e_1}{e_2},\\
\hsl(\ww{e_4}{e_2}) = \ww{e_1}{e_3},\\
\hsl(\ww{e_2}{e_3}) = \ww{e_1}{e_4},\label{eqn:minus}
\end{array}\right.
\eeqa
or in block matrix form,
\beqa
\label{eqn:Hbasis}
\hsl = \begin{bmatrix}
O & I\\
-I & O
\end{bmatrix},
\eeqa
where $I$ is the $3\times 3$ identity matrix.  By contrast, the Riemannian Hodge star \eqref{eqn:hs}, with  respect to the same basis \eqref{eqn:Hbasis0}, is
\beqa
\label{eqn:oring}
\hsr = \begin{bmatrix}
O & I\\
I & O
\end{bmatrix}\cdot
\eeqa
The difference is that, while $*$ is an involution, splitting $\Lambda^2$ into the direct sum $\Lambda^2 = \Lambda^+ \oplus \Lambda^-$ of \emph{self-dual} ($*\xi = \xi$) and \emph{anti-self-dual} ($*\xi = -\xi$) eigenvenctors, $\hsl$ is not (even though $\hsl$ is $\ipl{\,}{}$-self-adjoint).  Indeed, $\hsl^2 = -1.$  However, the complex multiplication on $\Lambda^2$ that the latter defines will be crucial to our construction of ``Riemannian Petrov Type" below\,---\,just as it was in the construction of Petrov Type in \cite{thorpe} (see also \cite[pp.~98ff.]{besse} and \cite[Chapter~5]{o1995}).  Let us describe this now.  Because $\hsl^2 =  -1$, one loses the self-dual/anti-self-dual splitting of $\Lambda^2$, but gains a complex structure in the process, via $i \defeq \hsl$.  Then the commuting condition
$
\cowl \circ \hsl = \hsl \circ \cowl,
$
where $\cowl$ is the Weyl curvature operator of $\gL$\,---\,which, being trace-free, automatically commutes with $\hsl$\,---\,turns $\cowl$ into a \emph{complex-linear} map on $\cx$.  The ``Petrov Type of $\gL$" is then precisely the complex eigenstructure of $\cowl$.  This eigenstructure was shown in \cite{thorpe} to coincide with the \emph{number} of critical points of $\gL$'s sectional curvature function.   (E.g., the Petrov Type of vacuum black holes is equivalent to their sectional curvature functions having exactly one ``spacelike" critical point at any point on $M$.)  Therefore, by having a Riemannian $\co$ commute with a Lorentzian $\hsl$, we can realize complex eigenstructures for $\co$ and thus ``Riemannian Petrov Types."  (Cf. with \cite{apostolov97}, wherein it was shown that an oriented Einstein 4-manifold will admit a (locally) Hermitian structure if and only if its self-dual Weyl tensor ${\cow}^+ \defeq \frac{1}{2}(\cow + \cow \circ *)$ has at least two of its three eigenvalues equal.)  As we do so, the following is helpful to keep in mind:
\begin{defn}
\label{defn:Pstuff}
Let \emph{$(M,\gL)$} be an oriented Lorentzian 4-manifold.  At any $p  \in M$, let $P$ denote an oriented 2-dimensional subspace of $T_pM$.  Then $P$ is \emph{nondegenerate} if the restriction of \emph{$\gL$} to $P$, \emph{$\gL|_P$}, is nondegenerate. The \emph{sign of $P$}, denoted \emph{$\vepo(P) = \pm1$}, is defined to be $-1$ if \emph{$\gL|_P$} is Lorentzian and $+1$ if \emph{$\gL|_P$} is positive-definite.  The 2-plane \emph{$\gL$-orthogonal} to a nondegenerate $P$, denoted by \emph{$\PL$}, is defined to be \emph{$$\PL  \defeq \hsl P.$$}Finally, following \cite{thorpe}, let $G_+(p) \cup G_-(p) \subseteq \Lambda^2(T_pM)$ denote the 2-Grassmannians of all decomposable 2-vectors of length $\pm 1$, respectively\emph{;} i.e., the set of all oriented, nondegenerate 2-dimensional subspaces of $T_pM$.  Note that \emph{$\vepo(P) = \ipl{P}{P}$} for any $P \in G_{\pm}(p)$.
\end{defn}
($\PL$ should be distinguished from its Riemannian counterpart $P^\perp \defeq *P$.  Also, recall that $\xi \in \Lambda^2$ is decomposable $\Leftrightarrow \ww{\xi}{\xi} = 0 \Leftrightarrow \ipl{\xi}{\hsl\,\xi} = 0$.) We close this section by analyzing the matrix of $\co$ with respect to \eqref{eqn:Hbasis0}.  Denoting by ``$R_{ijkl}$" the components of the Riemann curvature 4-tensor of $g$, we have (recall the minus sign in \eqref{eqn:minus0})
{\footnotesize
\beqa
\label{eqn:co_sum}
\co = -\!\begin{bmatrix}
R_{1212} & R_{1312} & R_{1412} & R_{3412} & R_{4212} & R_{2312}\\
R_{1213} & R_{1313} & R_{1413} & R_{3413} & R_{4213} & R_{2313}\\
R_{1214} & R_{1314} & R_{1414} & R_{3414} & R_{4214} & R_{2314}\\
R_{1234} & R_{1334} & R_{1434} & R_{3434} & R_{4234} & R_{2334}\\
R_{1242} & R_{1342} & R_{1442} & R_{3442} & R_{4242} & R_{2342}\\
R_{1223} & R_{1323} & R_{1423} & R_{3423} & R_{4223} & R_{2323}
\end{bmatrix},
\eeqa}
which, owing to the symmetry $R_{ijkl} = R_{klij}$, has the block form
\beqa
\label{eqn:co2}
\co = -\!\begin{bmatrix}
A & B\\
B^t & D
\end{bmatrix},
\eeqa
with $A$ and $D$ symmetric $3\times 3$ matrices and $B^t$ the transpose of the $3\times3$ matrix $B$, which is not symmetric in general.  Now, in order to motivate our construction in Section \ref{sec:psm} below, recall that in the Riemannian setting $\Ric = \lambda g$ if and only if $\co \circ \hsr = \hsr \circ \co$
which, as we now see from \eqref{eqn:oring} and \eqref{eqn:co2}, is the case if and only if $A=D$ and $B^t = B$.  Combining this with \eqref{eqn:gsec_def}, we thus have that
\beqa
\label{eqn:p2}
\underbrace{\,\gsec(P) = \gsec(P^{\perp})\,}_{\text{for any 2-plane $P$}} \iff \co \circ \hsr = \hsr \circ \co \iff \co = -\!\begin{bmatrix}
A & B\\
B & A
\end{bmatrix},
\eeqa
where $P^{\perp} \defeq \hsr P$.
If instead we had required that $\co$ and $\hsr$ \emph{anti}-commute, then $\gsec(P^{\perp})$ would have the opposite sign; in fact,
\beqa
\label{eqn:p2*}
\underbrace{\,\text{sec}(P) = -\text{sec}(P^{\perp})\,}_{\text{for any 2-plane $P$}} \iff \co \circ \hsr = -(\hsr \circ \co) \iff \co = -\!\begin{bmatrix}
A & B\\
-B & -A
\end{bmatrix}\cdot
\eeqa
We are now going to study geometries that are in some sense \emph{intermediate between \eqref{eqn:p2} and \eqref{eqn:p2*}}\,---\,see \eqref{eqn:p1} below.

\section{The case where $g$ is Riemannian and $h$ is Lorentzian}
\label{sec:psm}
The ``intermediate metrics" whose geometry we will be studying fit within the following larger class of ``almost-Einstein" metrics:
\begin{defn}[almost-Einstein\,metric]
\label{def:MgT}
An oriented Riemannian 4-manifold $(M,g)$ is \emph{almost-Einstein} if there is a nowhere vanishing vector field $T$ on $M$, together with local ordered frames $\{T,X_1,X_2,X_3\}$ in a neighborhood of each point, with respect to which the Ricci tensor \emph{$\text{Ric}_{\scalebox{0.6}{\emph{g}}}$} of $g$ uniformly takes the form
\beqa
\label{def:pE}
\emph{\text{$\text{Ric}_{\scalebox{0.6}{\emph{g}}}$}} = \begin{bmatrix}\lambda & \psi_1 & \psi_2 & \psi_3 \\ \psi_1 & \lambda & 0 & 0\\ \psi_2 & 0 & \lambda & 0\\ \psi_3 & 0 & 0 & \lambda  \end{bmatrix}\cdot
\eeqa
Here \emph{$\lambda \defeq \text{Ric}_{\scalebox{0.6}{\emph{g}}}(T,T)$} and $\lambda_i,\psi_i$ are smooth functions.
\end{defn}
Einstein metrics are locally, if not always globally, almost-Einstein.  Almost-Einstein metrics therefore deform the Einstein condition ``in the direction of a fixed vector field $T$."  Our motivation for considering them is that, on the 4-manifolds we'll be considering, non-flat Einstein metrics don't even exist, leaving metrics ``close" to Einstein as the only possibility:

\begin{lemma}
\label{lemma:berger}
Let $g$ be an almost-Einstein metric on a closed 4-manifold.  If $g$ is Einstein, then it is flat.
\end{lemma}

\begin{proof}
If the closed 4-manifold $M$ globally admits a nowhere vanishing vector field $T$, then it must have Euler characteristic zero: $\chi(M)=0$.  But by Berger \cite{berger}, on such a 4-manifold any Einstein metric $g$ must be flat; indeed, by the Chern-Gauss-Bonnet formula in dimension 4,
$$
\chi(M) = \frac{1}{8\pi^2}\int_M |W_{\scalebox{0.6}{\emph{g}}}|^2 - \frac{1}{2}|\cancelto{0}{\mathring{\text{Ric}}}_{\scalebox{0.6}{\emph{g}}}|^2 + \frac{1}{24}|\text{scal}_{\scalebox{0.6}{\emph{g}}}|^2,
$$
where $\mathring{\text{Ric}}_{\scalebox{0.6}{\emph{g}}} \defeq \text{Ric}_{\scalebox{0.6}{\emph{g}}} -  \frac{\text{scal}_{\scalebox{0.6}{\emph{g}}}}{4}g$ vanishes identically for Einstein metrics.
\end{proof}

On the other hand, let us observe in passing that almost-Einstein metrics certainly do exist locally:

\begin{lemma}
Let $\lambda,\psi_1,\psi_2,\psi_3$ be smooth functions on $\RR^4$ such that
$$
(\psi_1^2+\psi_2^2+\psi_3^2)^{\frac{1}{2}}\big|_{\bf 0} \neq \lambda\big|_{\bf 0} \neq 0
$$
at  the origin ${\bf 0} \in \RR^4$. Then there exists an almost-Einstein metric in a neighborhood of ${\bf 0}$.
\end{lemma}

\begin{proof}
The 2-tensor \eqref{def:pE} has eigenvalues
$$
\lambda \commas \lambda \commas \lambda \pm (\psi_1^2+\psi_2^2+\psi_3^2)^{\frac{1}{2}}.
$$
If $(\psi_1^2+\psi_2^2+\psi_3^2)^{\frac{1}{2}}\big|_{\bf 0} \neq \lambda\big|_{\bf 0}  \neq 0$, then these eigenvalues are all nonzero at ${\bf 0}$, in which case the 4-tensor defined in the coordinate frame on $\RR^4$ by \eqref{def:pE} is invertible at ${\bf 0}$.  As shown in \cite{deturck} (see also \cite{besse}), this guarantees the existence of a smooth Riemannian metric $g$ on a neighborhood of ${\bf 0}$ whose Ricci tensor is, with respect to the coordinate frame $\{\partial_1,\partial_2,\partial_3,\partial_4\}$, equal to the Ricci tensor in \eqref{def:pE}.   Setting $T \defeq \partial_1$ completes the proof.
\end{proof}

We now show that an important subclass of almost-Einstein metrics has a rich geometry intermediate between \eqref{eqn:p2} and \eqref{eqn:p2*}, characterized by \eqref{eqn:Weylpart}:

\begin{thm}
\label{prop:psE}
Let $(M,g)$ be an oriented Riemannian 4-manifold and $T$ a unit length vector field on $M$.  If the curvature operator of $g$ commutes with the Hodge star operator \emph{$\hsl$} of the Lorentzian metric \emph{$\gL \defeq g -2 T^\flat \otimes T^\flat$}, then $g$ is an almost-Einstein metric.  Furthermore, there exist bases with respect to which $W^+ = W^-$, where $W^{\pm}$ are the self-dual and anti-self-dual components of the Weyl curvature operator of $g$. In particular, if $M$ is compact, then it has Euler characteristic and signature both equal to zero.
\end{thm}

\begin{proof}
We will express the commuting condition
\beqa
\label{eqn:p00}
\co \circ \hsl = \hsl \circ \co
\eeqa
with respect to a local $g$-orthonormal frame of the form $\{T\defeq e_1,e_2,e_3,e_4\}$, where we recall that this basis is also $\gL$-orthonormal.  The matrices of $\hsl$ and $\co$ relative to the corresponding basis \eqref{eqn:Hbasis0} for $\Lambda^2$ were already given in \eqref{eqn:Hbasis} and \eqref{eqn:co2}, respectively.  It follows that \eqref{eqn:p00} holds if and only if $A = D$ and $B^t = -B$ in this basis:
\beqa
\label{eqn:p1}
\co \circ \hsl = \hsl \circ \co \iff \co = -\!\begin{bmatrix}
A & B\\
-B & A
\end{bmatrix}\cdot
\eeqa
(Note how this is intermediate between \eqref{eqn:p2} and \eqref{eqn:p2*}.)
We now determine the Ricci tensor of such a $g$.  From \eqref{eqn:co_sum} and \eqref{eqn:co2},
$B^t = -B$ yields
$$
R_{1334} = -R_{4212}\commas R_{1434}  = -R_{2312} \commas R_{1442} = -R_{2313}.
$$
(It also yields
$
R(e_3,e_4,T,e_2) = R(e_4,e_2,T,e_3) = R(e_2,e_3,T,e_4) = 0,
$
though we will not need these in what follows.) As a consequence,
\beqa
\arraycolsep=1.4pt\def\arraystretch{1.5}
\left\{\begin{array}{lr}
\text{Ric}_{12} = R_{3123}+R_{4124} = 2R_{3123},\nonumber\\
\text{Ric}_{13} = R_{2132}+R_{4134} = 2R_{2132},\nonumber\\
\text{Ric}_{14} = R_{2142}+R_{3143} = 2R_{2142},\nonumber
\end{array}\right.
\eeqa
which functions we'll denote by $\psi_1,\psi_2,\psi_3$, respectively.
On the other hand, the relation $A=D$ yields
\beqa
\arraycolsep=1.4pt\def\arraystretch{1.5}
\left\{\begin{array}{lr}
\text{Ric}_{23} = R_{1231}+R_{4234} = 0,\nonumber\\
\text{Ric}_{24} = R_{1241}+R_{3243} = 0,\nonumber\\
\text{Ric}_{34} = R_{1341}+R_{2342} = 0.\nonumber
\end{array}\right.
\eeqa
Next, the components
\beqa
\arraycolsep=1.4pt\def\arraystretch{1.5}
\left\{\begin{array}{lr}
\text{Ric}_{11} = R_{2112}+R_{3113}+R_{4114},\nonumber\\
\text{Ric}_{22} = R_{1221}+R_{3223}+R_{4224},\nonumber\\
\text{Ric}_{33} = R_{1331}+R_{2332}+R_{4334},\nonumber\\
\text{Ric}_{44} = R_{1441}+R_{2442}+R_{3443},\nonumber
\end{array}\right.
\eeqa
yield, together with $A=D$, the following three identities:
\beqa
\arraycolsep=1.4pt\def\arraystretch{1.5}
\left\{\begin{array}{lr}
\text{(1)}\ \ \text{Ric}_{11} - \text{Ric}_{22} +\text{Ric}_{33} - \text{Ric}_{44}= 2R_{1331}-2R_{2442} = 0,\nonumber\\
\text{(2)}\ \ \text{Ric}_{11} + \text{Ric}_{22} -\text{Ric}_{33} - \text{Ric}_{44}= 2R_{1221}-2R_{3443} = 0,\nonumber\\
\text{(3)}\ \ \text{Ric}_{11} - \text{Ric}_{22} -\text{Ric}_{33} + \text{Ric}_{44}= 2R_{4114}-2R_{2332} = 0.\nonumber\\
\end{array}\right.
\eeqa
The combinations $(1)+(2)$, $(1)+(3)$, and $(2)+(3)$ then yield, respectively,
$$
\text{Ric}_{11} = \text{Ric}_{44}  \comma \text{Ric}_{11} = \text{Ric}_{22} \comma \text{Ric}_{11} = \text{Ric}_{33}.
$$
Setting $\lambda \defeq \text{Ric}_{11}$ now puts the Ricci tensor of $g$, expressed  with respect to the orthonormal frame $\{T,e_2,e_3,e_4\}$, precisely in the form of \eqref{def:pE}.  (Note that because the traceless Ricci tensor does not vanish, one cannot use Schur's Lemma to prove that $\lambda$ must be constant.) Next, we show that $W^+ = W^-$.  To do so, let us keep $\{T\defeq e_1,e_2,e_3,e_4\}$, but instead of the basis \eqref{eqn:Hbasis0} for $\Lambda^2$, let us work instead with the basis $\{\xi_1^+,\xi_2^+,\xi_3^+,\xi_1^-,\xi_2^-,\xi_3^-\}$ of self-dual $(\xi_i^+)$ and anti-self-dual $(\xi_i^-)$ eigenvectors of the Riemannian Hodge star operator $\hsr$:
$$
\Big\{\underbrace{\,\frac{1}{\sqrt{2}}(\ww{e_1}{e_2}\pm\ww{e_3}{e_4})\,}_{\text{$\xi_1^{\pm}$}},\underbrace{\,\frac{1}{\sqrt{2}}(\ww{e_1}{e_3}\pm\ww{e_4}{e_2})\,}_{\text{$\xi_2^{\pm}$}},\underbrace{\,\frac{1}{\sqrt{2}}(\ww{e_1}{e_4}\pm\ww{e_2}{e_3})\,}_{\text{$\xi_3^{\pm}$}}\Big\}\cdot
$$
With respect to this basis, it can be shown (see, e.g., \cite[pp.~162-3]{PP}) that $\co$ takes the block form
\beqa
\label{eqn:Wpm0}
\co = \begin{bmatrix}
W^+ + \frac{\text{scal}_{\scalebox{0.6}{\emph{g}}}}{12}I & K\\
K^t & W^- + \frac{\text{scal}_{\scalebox{0.6}{\emph{g}}}}{12}I
\end{bmatrix},
\eeqa
where $\text{scal}_{\scalebox{0.5}{\emph{g}}}$ is the scalar curvature of $g$ and $W^{\pm}$ arise from its Weyl curvature operator $\cW\colon\Lambda^2\lra \Lambda^2$ as follows:
\beqa
\label{eqn:Wpm}
\cW = \begin{bmatrix}
W^+ & O\\
O & W^-
\end{bmatrix}\cdot
\eeqa
The equality $W^+ = W^-$ now follows from the fact that the Lorentzian Hodge star $\hsl$ satisfies $\hsl\,\xi_i^{\pm} = \pm\xi_i^{\mp}$, so that in this basis it takes the form
$$
\hsl = \begin{bmatrix}
O & -I\\
I & O
\end{bmatrix}\cdot
$$
The commuting condition $\co \circ \hsl = \hsl  \circ \co$ now directly yields $W^+ = W^-$.  Finally, that $M$ has signature zero, $\tau(M) = 0$, follows from the Thom-Hirzebruch signature formula: $\tau(M) = \frac{1}{12\pi^2}\int_M (|W^+|_{\scalebox{0.6}{\emph{g}}}^2 - |W^-|_{\scalebox{0.6}{\emph{g}}}^2)dV_{\scalebox{0.6}{\emph{g}}}$ (see, e.g., \cite[p.~371]{besse}.)  That $\chi(M) = 0$ follows, as was mentioned in  Lemma \ref{lemma:berger}, from the existence of a nowhere vanishing vector field on $M$.
\end{proof}

Let us pause to give the metrics of Theorem \ref{prop:psE} a more suggestive name:

\begin{defn}[$\hsl$-Einstein metric]
\label{def:*}
An oriented Riemannian 4-manifold $(M,g)$ is \emph{$\hsl$-Einstein} if there exists a unit length vector field $T$ on $M$ such that the curvature operator of $g$ commutes with the Hodge star operator \emph{$\hsl$} of the Lorentzian metric \emph{$\gL \defeq g -2 T^\flat \otimes T^\flat$}.
\end{defn}

As there are always two metrics under consideration here, $g$ and $\gL$, the following points are important to keep in mind:
\begin{enumerate}[leftmargin=*]
\item[i.] While $\co$ is self-adjoint with respect to its own $g$-induced inner product \eqref{eqn:ip} on $\Lambda^2$, it is not self-adjoint with respect to the $\gL$-induced inner product \eqref{eqn:ipL} on $\Lambda^2$.  Furthermore, while \eqref{eqn:ipL} is nondegenerate, it is not positive-definite: Its signature, as we remarked above, is $(-\!-\!-\!+\!++)$.
\item[ii.] For this reason, the components of $\co$ do not, in general, satisfy the algebraic Bianchi identity with respect to $\ipl{\,}{}$, even for $g$-orthonormal bases of the  form $\{T\defeq e_1,e_2,e_3,e_4\}$.  Indeed, with respect to such a basis we have, by \eqref{eqn:minus0},
$$
\co = -\!\sum_{k<l}R_{ijkl}\,\ww{e_k}{e_l},
$$
but the permuted sum
$$
\hspace{.25in}\underbrace{\,\ipl{\co(\ww{e_i}{e_j})}{\ww{e_k}{e_l}}\,}_{\text{$-R_{ijkl}\,\vepo(\ww{e_k}{e_l})$}}+ \underbrace{\,\ipl{\co(\ww{e_j}{e_k})}{\ww{e_i}{e_l}}\,}_{\text{$-R_{jkil}\,\vepo(\ww{e_i}{e_l})$}} + \underbrace{\,\ipl{\co(\ww{e_k}{e_i})}{\ww{e_j}{e_l}}\,}_{\text{$-R_{kijl}\,\vepo(\ww{e_j}{e_l})$}}
$$
need not vanish in general, since the $\vepo$-terms need not all have the same sign.  (By contrast, the Lorentzian curvature operator is \emph{defined} by $({\text{Rm}_{\scalebox{0.4}{\emph{L}}}})_{ijkl} \defeq -\ipl{\col(\ww{e_i}{e_j})}{\ww{e_k}{e_l}}$; i.e., without $\vepo(\ww{e_k}{e_l})$.  Thus 
$$
\hspace{.25in}\underbrace{\,\ipl{\col(\ww{e_i}{e_j})}{\ww{e_k}{e_l}}\,}_{\text{$({\text{Rm}_{\scalebox{0.4}{\emph{L}}}})_{ijkl}$}} + \underbrace{\,\ipl{\col(\ww{e_j}{e_k})}{\ww{e_i}{e_l}}\,}_{\text{$({\text{Rm}_{\scalebox{0.4}{\emph{L}}}})_{jkil}$}} + \underbrace{\,\ipl{\col(\ww{e_k}{e_i})}{\ww{e_j}{e_l}}\,}_{\text{$({\text{Rm}_{\scalebox{0.4}{\emph{L}}}})_{kilj}$}}
$$
will always vanish by the algebraic Bianchi identity for $\text{Rm}_{\scalebox{0.4}{\emph{L}}}$.)
\end{enumerate}

Both of these facts will have important consequences.   In any case, as $\hsl^2 = -1$, we follow \cite{thorpe} and define complex scalar multiplication on $\Lambda^2$ by
\beqa
\label{eqn:complex}
i\xi \defeq \hsl\,\xi.
\eeqa
As a result, the real $6$-dimensional vector space $\Lambda^2$ becomes a $3$-dimensional complex vector space, which we'll henceforth denote by $\cx$. We may take as a basis for $\cx$ the first three elements in \eqref{eqn:Hbasis0}, since now
$$
\ww{e_3}{e_4} = -i\,\ww{e_1}{e_2} \comma \ww{e_4}{e_2} = -i\,\ww{e_1}{e_3} \comma \ww{e_2}{e_3} = -i\,\ww{e_1}{e_4}.
$$
More importantly, for $\hsl$-Einstein metrics, the condition $\co \circ \hsl = \hsl \circ \co$ now renders $\co$ into a \emph{complex-linear} map on $\cx$:
$$
\co(i\xi) =  \co(\hsl(\xi)) =  \hsl(\co(\xi)) = i\co(\xi).
$$
Now, if we wish to characterize $\hsl$-Einstein metrics as in \eqref{eqn:gsec_def}, we will need to modify $g$'s sectional curvature $\gsec$ to one that is ``adapted" to $\ipl{\,}{}$\,---\,the reason for this is because the self-adjointness of $\hsr$ with respect to $\ipr{\,}{}$ is crucial in establishing \eqref{eqn:gsec_def}, but in the case at hand $\hsl$ is not self-adjoint with respect to $\ipr{\,}{}$. Thus, we'll need the following (recall Definition \ref{defn:Pstuff}):

\begin{defn}[$L$-sectional curvature]
\label{def:T_sec0}
Let $(M,g)$ be an oriented Riemannian 4-manifold with  curvature operator \emph{$\co$} and $T$ a unit length vector field on $M$. Consider the Lorentzian metric \emph{$\gL \defeq g-2T^\flat \otimes T^\flat$} and its inner product \emph{$\ipl{\,}{}$} on $\Lambda^2$.  Then the function \emph{$\Tsec$} defined on each $G_+(p) \cup G_-(p) \subseteq T_pM$ by
\beqa
\label{def:T_sec}
\text{\emph{$\Tsec(P) \defeq \vepo(P)\ipl{\co P}{P}$}},
\eeqa
is called the $L$-\emph{sectional curvature of $g$}.
\end{defn}
Observe that $\Tsec$ is not defined for \emph{degenerate} 2-planes $P$ (i.e., those 2-dimensional subspaces $P \subseteq T_pM$ on which $\gL|_P$ is degenerate).  On the other hand, there are 2-planes $P$ at which $\Tsec$ and $\gsec$ agree:
\begin{lemma}
If $P \in G_+(p) \cup G_-(p)$ contains $T$ or is orthogonal to $T$, then \emph{$\gsec(P) = \Tsec(P)$}.
\end{lemma}

\begin{proof}
If $P \in G_+(p) \cup G_-(p)$ contains $T$, then we may write $P = \ww{T}{w}$ for some $\gL$-unit $w$ that  is $\gL$-orthogonal to $T$ ($w$ must be spacelike; see, e.g., \cite[Lemma~26,~p.~141]{o1983}).  Now form a $\ipl{\,}{}$-orthonormal basis \eqref{eqn:Hbasis0} for $\Lambda^2$ with $T \defeq e_1$ and  $w \defeq e_2$, and note that
$$
\vepo(P) = \vepo(\ww{e_1}{e_2}) = \ipl{\ww{e_1}{e_2}}{\ww{e_1}{e_2}} = -1.
$$
Since this basis is also $\ipr{\,}{}$-orthonormal by construction, we have that $\co(\ww{e_1}{e_2}) = -R_{1212}\,\ww{e_1}{e_2} + \cdots$.  Then
$$
\Tsec(P) = \vepo(P)\ipl{\co(\ww{e_1}{e_2})}{\ww{e_1}{e_2}} = -R_{1212} = R_{1221} \overset{\eqref{sec:v0}}{=} \gsec(P).
$$
If $P$ is orthogonal to $T$, then choose a basis \eqref{eqn:Hbasis0} with $T \defeq e_1$ and $P = \ww{e_2}{e_3}$, where, in this case, $\vepo(P) = \ipl{\ww{e_2}{e_3}}{\ww{e_2}{e_3}} = +1$.  Writing $\co(\ww{e_2}{e_3}) = \cdots - R_{2323}\,\ww{e_2}{e_3} + \cdots$, we have
$$
\Tsec(P) = \vepo(P)\ipl{\co(\ww{e_2}{e_3})}{\ww{e_2}{e_3}} = -R_{2323} = R_{2332} = \gsec(P),
$$
which completes the proof.
\end{proof}

The virtue of working with $\Tsec$ is that it recovers something of the $\gsec$-Einstein characterization \eqref{eqn:gsec_def}, thus providing an alternative notion of ``almost-Einstein" metric (in what follows, note the crucial role played by the self-adjointness of $\hsl$ with respect to $\ipl{\,}{}$):

\begin{thm}[almost-Einstein metric, Lorentzian case]
\label{thm:secsec}
If an oriented Riemannian 4-manifold $(M,g)$ is \emph{$\hsl$-Einstein}, then 
$$
\text{\emph{$\Tsec(\PL) = \Tsec(P)$}}
$$
for all $P \in G_+(p) \cup G_-(p)$.
\end{thm}

\begin{proof}
Our proof follows that of \cite[Theorem,~p.~5]{thorpe} (which Theorem is with respect to the Lorentzian curvature operator, not the Riemannian one that we are dealing with here; see our remark after the proof).  Suppose that $g$ is a $\hsl$-Einstein metric with respect to some choice of unit length vector field $T$ on $M$. Noting that $\vepo(\PL) = - \vepo(P)$ (this follows from the self-adjointness of $\hsl$), we have that
\beqa
\Tsec(\PL) \!\!&=&\!\! \vepo(\PL)\ipl{\co(\hsl P)}{\hsl P}\label{eqn:secsec}\\
&\overset{\eqref{eqn:p00}}{=}&\!\! -\vepo(P)\ipl{\hsl(\co P)}{\hsl P}\nonumber\\
&=&\!\! \vepo(P)\ipl{\co P}{P}\label{eqn:sa_1}\\
&=&\!\! \Tsec(P),\nonumber
\eeqa
where in \eqref{eqn:sa_1} we used the self-adjointness of $\hsl$ once again.
\end{proof}

It's natural to ask whether the converse is true\,---\,as it is in the fully Riemannian \cite{thorpe2} or fully Lorentzian \cite{thorpe} cases.  In fact it is not true here. Indeed, suppose that $\Tsec(\PL) = \Tsec(P)$ for all nondegenerate 2-planes $P$, and that we would like to conclude from this that $\hsl \circ \co = \co \circ \hsl$ (equivalently, that $\hsl \circ \co \circ \hsl = -\co$).  Using the self-adjointness of $\hsl$ once again, we would have
\beqa
\label{eqn:cont}
\underbrace{\,-\vepo(P) \ipl{\hsl \co \hsl\!P}{P}\,}_{\text{$\Tsec(\PL)$, via \eqref{eqn:secsec}}}\hspace{-.08in}\,\, = \underbrace{\,\vepo(P)\ipl{\co P}{P}\,}_{\text{$ \Tsec(P)$}}.
\eeqa
This implies that the operators $\hsl \circ\co \circ \hsl$ and $-\co$ have the same $\ipl{\,}{}$-quadratic forms on all nondegenerate 2-planes $P$.  Now, if two curvature operators both satisfy the algebraic Bianchi identity and have equivalent quadratic forms, then by a standard argument they must be equal (even if their quadratic forms are equal only on \emph{nondegenerate} 2-planes, by a continuity argument they would still be equal on all 2-planes; see \cite{thorpe}).  However, $\co$ does not satisfy the algebraic Bianchi identity with respect to $\ipl{\,}{}$, hence neither does $\hsl \circ \co \circ \hsl$. Moving on, we are now ready to classify $\hsl$-Einstein metrics, by defining a Riemannian version of Petrov Type:

\begin{defn}[Petrov Type of $\hsl$-Einstein metric]
\label{def:PT1}
Let $(M,g)$ be \emph{$\hsl$}-Einstein, with complex-linear curvature operator \emph{$\co\colon \cx \lra \cx$}.  Then $g$ has \emph{Petrov Type I, II, or III at $p \in M$} if \emph{$\co$} has, respectively, 3, 2, or 1 independent complex eigenvectors at $p$.
\end{defn}

Three distinct eigenvalues is the generic case; the remaining cases are ``algebraically special." Jordan normal-form theory now yields the following matrix characterization of Petrov Type (cf. \cite[Theorem,~p.~3]{thorpe}):
\begin{thm}
\label{thm:JordanP}
Let $(M,g)$ be \emph{$\hsl$}-Einstein.  At each $p \in M$, its complex-linear curvature operator \emph{$\co\colon \cx \lra \cx$} will be similar to one of
\beqa
\underbrace{\,\begin{bmatrix}
\lambda_1  & 0 & 0\\
0 & \lambda_2 & 0\\
0 & 0 & \lambda_3
\end{bmatrix}\,}_{\emph{\text{Type I}}}\comma
\underbrace{\,\begin{bmatrix}
\lambda_1  & 0 & 0\\
0 & \lambda_2 & 1\\
0 & 0 & \lambda_2
\end{bmatrix}\,}_{\emph{\text{Type II}}}\comma
\underbrace{\,\begin{bmatrix}
\lambda_1  & 1 & 0\\
0 & \lambda_1 & 1\\
0 & 0 & \lambda_1
\end{bmatrix}\,}_{\emph{\text{Type III}}},
\eeqa
with $\lambda_1, \lambda_2, \lambda_3 \in \mathbb{C}$ not necessarily distinct. 
\end{thm}

\begin{proof}
If $\co$ has three linearly independent eigenvectors (Type I) at $p\in M$, then it is diagonalizable, with eigenvalues $\lambda_1,\lambda_2,\lambda_3 \in \mathbb{C}$ not necessarily distinct; this is the leftmost matrix.  If $\co$ has two linearly independent eigenvectors (Type II), then it is not diagonalizable and splits into two cases: 1) Two of its eigenvalues are equal, or 2) all three are equal. Since in Type II the total geometric multiplicity is 2, either case leads to two Jordan blocks and thus to the middle matrix above (if all three eigenvalues are equal, then $\lambda_1 = \lambda_2$). The final case, with only one linearly independent eigenvector (Type III), has geometric multiplicity 1, hence one Jordan block. 
\end{proof}

We would now like to relate a $\hsl$-Einstein metric's Petrov Type to its sectional curvature $\gsec$; before doing so, however, we will first prove a more general result by characterizing the critical points of $\Tsec$ (the reason for taking such a path will appear in Theorem \ref{thm:SS} below).  As we do so, recall once again that $\co$ is not self-adjoint with respect to $\ipl{\,}{}$.  In practice, this means that in an arbitrary $\gL$-orthonormal frame $\{e_1,e_2,e_3,e_4\}$ with timelike direction $e_1$ (we don't assume that $T = e_1$), the components
\beqa
\label{eqn:K}
K_{ijkl} \defeq -\ipl{\co(\ww{e_i}{e_j})}{\ww{e_k}{e_l}},
\eeqa
while satisfying $K_{jikl} = -K_{ijkl}$ and $K_{ijlk} = - K_{ijkl}$, are however not pairwise symmetric in general: $K_{ijkl} \neq K_{klij}$.  But if $e_1 = T$, then we can say more:

\begin{lemma}
\label{lemma:crit2}
With respect to any oriented \emph{$\gL$}-orthonormal basis of the form $\{T\defeq e_1,e_2,e_3,e_4\}$, the components \eqref{eqn:K} satisfy
\beqa
\label{eqn:the_same0}
\text{\emph{$K_{ijkl} = R_{ijkl}\,\vepo(\ww{e_k}{e_l})$}}.
\eeqa
Hence in such a basis,
\beqa
\label{eqn:the_same}
\text{\emph{$K_{ijkl} = K_{klij} \iff \vepo(\ww{e_i}{e_j}) = \vepo(\ww{e_k}{e_l})$}}.
\eeqa

\begin{proof}
The key is that such a basis is also $g$-orthonormal, hence by  \eqref{eqn:minus0},
$$
K_{ijkl} = -\ipl{\!\!\!\!\!\!\underbrace{\co(\ww{e_i}{e_j})}_{\text{$-R_{ijkl}\,\ww{e_k}{e_l} + \cdots$}}\!\!\!\!\!\!}{\ww{e_k}{e_l}} = R_{ijkl}\,\vepo(\ww{e_k}{e_l}).
$$
Since $R_{ijkl} = R_{klij}$, it's clear that $K_{ijkl} = K_{klij}$ if and only if $\vepo(\ww{e_i}{e_j}) = \vepo(\ww{e_k}{e_l})$.
\end{proof}

\end{lemma}

Now we are in a position to characterize the critical points of $\Tsec$:

\begin{prop}
\label{prop:crit}
Let $(M,g)$ be an oriented Riemannian 4-manifold and $T$ a unit length vector field on $M$.  Let \emph{$\Tsec$} be  the $L$-sectional curvature of $g$ with respect to \emph{$\gL \defeq g - 2T^\flat \otimes T^\flat$}.  Then at any $p \in M$, $P \in G_+(p)$ is a critical point of \emph{$\Tsec$} if and only if with respect to any oriented orthonormal frame $\{e_1,e_2,e_3,e_4\}$ of $T_pM$ with timelike direction $e_1$ and with $P = \ww{e_3}{e_4}$,
\beqa
\label{eqn:KK+}
K_{34ij} = -K_{ij34} \comma \text{$(i,j) \neq (3,4)$ or $(1,2)$},
\eeqa
with $K_{ijkl}$ given by \eqref{eqn:K}. 
Similarly, for $P \in G_-(p)$ with $P = \ww{e_1}{e_2}$, 
\beqa
\label{eqn:KK-}
K_{12ij} = -K_{ij12} \comma \text{$(i,j) \neq (3,4)$ or $(1,2)$}.
\eeqa
If $e_1 = T$, then $P = \ww{e_3}{e_4}$ is a critical point of \emph{$\Tsec \Leftrightarrow R_{3442} = R_{3423} = 0$}, and $P = \ww{e_1}{e_2}$ is a critical point of \emph{$\Tsec \Leftrightarrow R_{1213} = R_{1214} = 0$}.
\end{prop}

\begin{proof}
We employ the same coordinate chart $\phi$ as in the proof of \cite[Lemma,~p.~5]{thorpe}, though  with minor index changes.  Given a critical 2-plane $P \in G_+(p)$, we may express it as $P = \ww{e_3}{e_4}$ for two spacelike $\gL$-orthonormal vectors $e_3,e_4$, and then extend these to a $\gL$-orthonormal basis $\{e_1,e_2,e_3,e_4\}$ for $T_pM$ with $e_1$ timelike (we do not assume that $e_1=T$, or even that $g$ is $\hsl$-Einstein).  Now define a coordinate chart $\phi$ from $\RR^4$ into $G_+(p)$ by
$$
\phi(x_1,x_2,x_3,x_4) \defeq \frac{\ww{(e_3 + x_3e_1+x_4e_2)}{(e_4+ x_1e_1 + x_2e_2)}}{\sqrt{1-x_1^2+x_2^2-x_3^2+x_4^2-(x_2x_3-x_1x_4)^2}}\cdot
$$
(Note that the 2-Grassmannians $G_{\pm}(p) \subseteq \Lambda^2$ each have dimension $2(4-2) =4$. Note also that $\phi({\bf 0})  = P$, that $\phi$ is a diffeomorphism on a neighborhood of ${\bf 0} \in \RR^4$, and that $\ipl{\phi(x)}{\phi(x)} = +1$ for every $x = (x_1,x_2,x_3,x_4)$ in  its domain.)  Then $P$ is a critical point of $\Tsec$ if and only if
$$
\frac{\partial (\Tsec \circ \phi)}{\partial x_i}\bigg|_{\bf 0} = 0 \comma i=1,2,3,4.
$$
As we now show, each partial derivative will yield one of the cases in \eqref{eqn:KK+}.  To compute the case $i=1$, let us set $x_2=x_3=x_4=0$ from the outset, which simplifies the computation:
\beqa
\underbrace{\,\frac{\partial}{\partial x_1}\bigg|_{\bf 0}\!\!(\Tsec \circ \phi){\scriptstyle \Big|_{\text{$(x_1,0,0,0)$}}}\,}_{0} \!\!\!\!\!\!\!\!&=&\!\!\!\! \frac{\partial}{\partial x_1}\bigg|_{\bf 0}\Tsec\bigg(\frac{\ww{e_3}{(e_4+ x_1e_1)}}{\sqrt{1-x_1^2}}\bigg)\nonumber\\
&=&\!\!\!\! \frac{\partial}{\partial x_1}\bigg|_{\bf 0}\frac{\ipl{\co(\ww{e_3}{e_4})+x_1\co(\ww{e_3}{e_1})}{\ww{e_3}{(e_4+x_1e_1)}}}{1-x_1^2}\nonumber\\
&\overset{\eqref{eqn:K}}{=}&\!\!\!\! -\frac{\partial}{\partial x_1}\bigg|_{\bf 0} \frac{K_{3434}+x_1(K_{3431}+K_{3134})+x_1^2K_{3131}}{1-x_1^2}\nonumber\\
&=&\!\!\!\! -(K_{3431}+K_{3134}),\nonumber
\eeqa
so that $K_{3431} =  -K_{3134}$.  For $i=2$,
\beqa
\underbrace{\,\frac{\partial}{\partial x_2}\bigg|_{\bf 0}\!\!(\Tsec \circ \phi){\scriptstyle \Big|_{\text{$(0,x_2,0,0)$}}}\,}_{0} \!\!\!\!\!\!\!\!&=&\!\!\!\! \frac{\partial}{\partial x_2}\bigg|_{\bf 0}\Tsec\bigg(\frac{\ww{e_3}{(e_4+ x_2e_2)}}{\sqrt{1+x_2^2}}\bigg)\nonumber\\
&=&\!\!\!\! \frac{\partial}{\partial x_2}\bigg|_{\bf 0}\frac{\ipl{\co(\ww{e_3}{e_4})+x_2\co(\ww{e_3}{e_2})}{\ww{e_3}{(e_4+x_2e_2)}}}{1+x_2^2}\nonumber\\
&\overset{\eqref{eqn:K}}{=}&\!\!\!\! -\frac{\partial}{\partial x_2}\bigg|_{\bf 0} \frac{K_{3434}+x_2(K_{3432}+K_{3234})+x_2^2K_{3232}}{1+x_2^2}\nonumber\\
&=&\!\!\!\! -(K_{3432}+K_{3234}).\nonumber
\eeqa
The cases $i=3,4$ likewise yield $K_{3414} = -K_{1434}$ and $K_{3424}= -K_{2434}$, respectively.  Similar computations with $
\phi \defeq \frac{\ww{(e_1 + x_3e_3+x_4e_4)}{(e_2+ x_1e_3 + x_2e_4)}}{\sqrt{1+x_1^2+x_2^2-x_3^2-x_4^2-(x_2x_3-x_1x_4)^2}}
$ yield  \eqref{eqn:KK-} for $P = \ww{e_1}{e_2} \in G_-(p)$. Finally, if a critical 2-plane $P \in G_+(p)$ is  orthogonal to $T$, meaning that it can be expressed as $P = \ww{e_3}{e_4}$ with respect to an oriented $\gL$-orthonormal basis $\{T = e_1,e_2,e_3,e_4\}$, then by Lemma \ref{lemma:crit2}
$$
K_{4234} \overset{\eqref{eqn:the_same}}{=} K_{3442} \overset{\eqref{eqn:KK+}}{=} -K_{4234} \imp R_{4234} \overset{\eqref{eqn:the_same0}}{=}  0.
$$
Likewise, $R_{3423} = 0$.  Similarly, if a critical 2-plane $P  \in G_-(p)$ can be expressed as $P = \ww{e_1}{e_2}$ with respect to $\{T = e_1,e_2,e_3,e_4\}$\,---\,i.e., if $P$ contains $T$\,---\,then
$$
K_{1213} \overset{\eqref{eqn:the_same}}{=} K_{1312} \overset{\eqref{eqn:KK-}}{=} -K_{1213} \imp R_{1213} \overset{\eqref{eqn:the_same0}}{=}  0.
$$
Likewise, $R_{1214} = 0$.  This completes the proof.
\end{proof}

In general, critical points of $\Tsec$ and $\gsec$ needn't coincide.  That is because, as shown in \cite{thorpe2}, any 2-plane $P$ will be critical for $\gsec$ if and only if
\beqa
\label{eqn:thorpecrit}
\co P = aP+b(\hsr\hspace{.01in}P)
\eeqa
for some $a,b \in \RR$ (this also follows from Proposition \ref{prop:crit}, by replacing $\Tsec$ with $\gsec$ and $K_{ijkl}$ with $R_{ijkl}$).  However, \eqref{eqn:thorpecrit} does suggest a relationship with the \emph{eigenvectors} of $\co\colon \cx \lra \cx$, and thus the Petrov Type of $g$:

\begin{thm}
\label{thm:n}
Let $(M,g)$ be a \emph{$\hsl$}-Einstein manifold.  If $P \in G_{+}(p) \cup G_-(p)$ is orthogonal to $T$ or contains $T$, then $P$ is an eigenvector of \emph{$\co\colon \cx \lra \cx$} if and only if $P$ is a critical point of \emph{$\gsec$}.
\end{thm}

\begin{proof}
Suppose $P \in G_+(p)$ is orthogonal to $T$ and write $P = \ww{e_3}{e_4}$, where $e_3,e_4$ are part of an oriented $g$- and $\gL$-orthonormal basis $\{T\defeq e_1,e_2,e_3,e_4\}$; by \eqref{eqn:Hbasis} and \eqref{eqn:oring}, $i P = \hsl P = \ww{e_1}{e_2} = \hsr P$.  Hence, if $P$ is an eigenvector of $\co$ with eigenvalue $a+ib \in \mathbb{C}$, then
$$
\co P = (a+ib)P = aP+b(\hsr P).
$$
By \eqref{eqn:thorpecrit}, $P$ is therefore a critical point of $\gsec$; the converse is obvious.  (In fact $P$ is also critical for $\Tsec$.) Similarly, if $P \in G_-(p)$ contains $T$ and is an eigenvector of $\co$, then writing $P = \ww{e_1}{e_2}$, we would now have $\co P = (a+ib)P = aP-b(\hsr P)$.
\end{proof}

Let us now generalize $\hsl$-Einstein metrics by extending Definition \ref{def:PT1} to \emph{any} oriented Riemannian 4-manifold admitting a nowhere vanishing vector field (hence $\chi(M) = 0$ if $M$ is compact).  This is easy to accomplish: For any choice of Lorentzian $\hsl$, the curvature operator $\co$ of $g$ will decompose as
\beqa
\label{eqn:SS}
\co = \underbrace{\,\frac{1}{2}(\co - \hsl \circ \co \circ \hsl)\,}_{\text{``$\cS\,$"}} + \underbrace{\,\frac{1}{2}(\co + \hsl \circ \co \circ \hsl)\,}_{\text{``$\cA$"}},
\eeqa
where the ``symmetric" and ``anti-symmetric" operators $\cS$ and $\cA$ satisfy
$\cS \circ \hsl = \hsl \circ \cS$ and $\cA \circ \hsl = -\hsl \circ \cA$, respectively.  In particular, $\cS$ is a complex-linear map on $\cx$, and $\co$ is $\hsl$-Einstein if and only if $\co = \cS$.
(As remarked in the Introduction, for $\co, \hsr$ we have $\cS \defeq \frac{1}{2}(\co + \hsr \circ \co \circ \hsr) = \cW + \frac{\text{scal}_{\scalebox{0.5}{\emph{g}}}}{12}I$; recall \eqref{eqn:Wpm0}-\eqref{eqn:Wpm}.)  It is now an easy matter to extend Definition \ref{def:PT1} to any oriented Riemannian 4-manifold $(M,g)$ with $\chi(M) = 0$:

\begin{defn}[Petrov Type, general version]
\label{def:PT2}
Let $(M,g)$ be an oriented Riemannian 4-manifold with curvature operator \emph{$\co$}, let $T$ be a unit length vector field on $M$, and let \emph{$\hsl$} be the Hodge star operator of the  Lorentzian metric \emph{$\gL \defeq g - 2T^{\flat} \otimes T^{\flat}$}.  At any $p \in M$, $g$ has \emph{Petrov Type I, II, or III} if the operator \emph{$\cS \defeq \frac{1}{2}(\co - \hsl \circ \co \circ \hsl)$}, viewed as a complex-linear map \emph{$\cS\colon \cx\lra \cx$}, has 3, 2, or 1 independent eigenvectors at $p$, respectively.
\end{defn}

\begin{thm}
\label{thm:SS}
Defining the \emph{$L$-sectional curvature of} \emph{$\cS$}, in analogy with \eqref{def:T_sec}, to be
$$
\text{\emph{$\Ssec(P) \defeq \vepo(P)\ipl{\cS P}{P},$}}
$$
the analogues of Theorems \ref{thm:secsec} and \ref{thm:JordanP} remain true with respect to \emph{$\cS$} and \emph{$\Ssec$}.  Furthermore, if $P \in G_+(p) \cup G_-(p)$ is an eigenvector of \emph{$\cS\colon \cx \lra \cx$} that is orthogonal to $T$ or contains $T$, then $P$ is a critical point of \emph{$\Ssec$}.
\end{thm}

\begin{proof}
The analogues of Theorems \ref{thm:secsec} and \ref{thm:JordanP} follow exactly as before. As for the eigenvectors of $\cS$, first observe that the components $S_{ijkl}$ of $\cS$ satisfy
\beqa
S_{ijkl} \!\!&\defeq&\!\! -\ipl{\cS(\ww{e_i}{e_j})}{\ww{e_k}{e_l}}\nonumber\\
&=&\!\! -\frac{1}{2}\ipl{(\co - \hsl \circ \co \circ \hsl)(\ww{e_i}{e_j})}{\ww{e_k}{e_l}}\nonumber\\
&\overset{\eqref{eqn:K}}{=}&\!\! \frac{1}{2}\Big(K_{ijkl} + \underbrace{\,\ipl{(\co \circ \hsl)(\ww{e_i}{e_j})}{\hsl(\ww{e_k}{e_l})}\,}_{\text{``$-\epsilon_{ij}\epsilon_{kl}K_{\hsl(ij)\hsl(kl)}$"}}\Big),\nonumber
\eeqa
where $\epsilon_{ab} \defeq \vepo(\ww{e_a}{e_b})$. Thus
\beqa
\label{eqn:Scomp}
S_{ijkl} = \frac{1}{2}(K_{ijkl}-\epsilon_{ij}\epsilon_{kl}K_{\hsl(ij)\hsl(kl)}).
\eeqa
Now suppose that $P \in G_+(p)$ is an eigenvector of $\cS$ that is $\gL$-orthogonal to $T$, and once again write $T \defeq e_1$ and $P = \ww{e_3}{e_4}$ as before. If $P$ has eigenvalue $a+ib \in \mathbb{C}$, then
$$
\cS(\ww{e_3}{e_4}) = (a+ib)(\ww{e_3}{e_4}) = b (\ww{e_1}{e_2}) + a(\ww{e_3}{e_4}),
$$
so that
$$
S_{3412} = b \commas S_{3434} = -a \commas S_{3413} = S_{3414} = S_{3442} = S_{3423} = 0.
$$
Since Proposition \ref{prop:crit} holds with $\Ssec$ in place of $\Tsec$ and $S_{ijkl}$ in place of $K_{ijkl}$, to show that $P$ is a critical point of $\Ssec$ we must show that
$$
S_{1334} = S_{1434} = S_{4234} = S_{2334} = 0.
$$
Consider $S_{1334}$; via \eqref{eqn:Scomp} and Lemma \ref{lemma:crit2},
$$
S_{3413} = \frac{1}{2}(\,\underbrace{K_{3413}\,}_{\text{$-R_{3413}$}}+ \underbrace{K_{1242}}_{\text{$R_{1242}$}}\,) = 0 \imp R_{3413} = R_{1242}.
$$
Thus
$$
S_{1334} = \frac{1}{2}(\,\underbrace{K_{1334}\,}_{\text{$R_{1334}$}}+ \underbrace{K_{4212}}_{\text{$-R_{4212}$}}\,) = \frac{1}{2}(R_{1334} - R_{4212}) = 0.
$$
Similarly, $S_{1434} = 0$.  For $S_{4234}$,
$$
S_{3442} = \frac{1}{2}(\,\underbrace{K_{3442}\,}_{\text{$R_{3442}$}} - \underbrace{K_{1213}}_{\text{$-R_{1213}$}}\,) = 0 \imp R_{3442} = -R_{1213}.
$$
Thus
$$
S_{4234} = \frac{1}{2}(\,\underbrace{K_{4234}\,}_{\text{$R_{4234}$}} -\underbrace{K_{1312}}_{\text{$-R_{1312}$}}\,) = \frac{1}{2}(R_{4234} + R_{1312}) = 0.
$$
Similarly, $S_{2334} = 0$.  We conclude that any $\cS$-eigenvector $P \in G_+(p)$ that is orthogonal to $T$ is a critical point of $\Ssec$.  The case when $P \in G_-(p)$ contains $T$ is similar, but now with $S_{1312}=S_{1412}=S_{4212}=S_{2312}=0$.
\end{proof}

\section{The case where $g$ is Riemannian and $h$ is split-signature}
\label{sec:ssm}

We now consider metrics $h$ of signature $(-\!-\!++)$ (cf. \cite{CP, derdzinski}).  Note that, unlike the Lorentzian case, $M$ need not have Euler characteristic zero to support such metrics; e.g., we can take any pair of oriented closed Riemannian 2-manifolds $(M_1,g_1), (M_2,g_2)$ and define $h\defeq (-g_1)\oplus g_2$ on $M_1 \times M_2$.  Observe that $h$ and $g \defeq g_1\oplus g_2$ satisfy the following property:

\begin{defn}
Let $(M,g)$ be an oriented Riemannian 4-manifold.  A split-signature metric $h$ on $M$ is a \emph{deformation of $g$} if $g$ and $h$ have local orthonormal frames in common about every point of $M$.
\end{defn}

As for the Hodge  star operator $\hsh$ of a split-signature metric $h$, for any oriented local $h$-orthonormal frame $\{e_1,e_2,e_3,e_4\}$ with $e_1,e_2$ both timelike, the action of $\hsh$ on the corresponding basis \eqref{eqn:Hbasis0} for $\Lambda^2$ is
\beqa
\arraycolsep=1.4pt\def\arraystretch{1.5}
\left\{\begin{array}{lr}
\hsh(\ww{e_1}{e_2}) = +\ww{e_3}{e_4},\\
\hsh(\ww{e_1}{e_3}) = -\ww{e_4}{e_2},\\
\hsh(\ww{e_1}{e_4}) = -\ww{e_2}{e_3},
\end{array}\right.  \comma \arraycolsep=1.4pt\def\arraystretch{1.5}
\left\{\begin{array}{lr}
\hsh(\ww{e_3}{e_4}) = +\ww{e_1}{e_2},\\
\hsh(\ww{e_4}{e_2}) = -\ww{e_1}{e_3},\\
\hsh(\ww{e_2}{e_3}) = -\ww{e_1}{e_4},\nonumber
\end{array}\right.
\eeqa
or in matrix form,
{\small
\beqa
\label{eqn:Hhbasis}
\hsh = \begin{bmatrix}
0 & 0 & 0 & 1 & 0 & 0\\
0 & 0 & 0 & 0 & -1 & 0\\
0 & 0 & 0 & 0 & 0 & -1\\
1 & 0 & 0 & 0 & 0 & 0\\
0 & -1 & 0 & 0 & 0 & 0\\
0 & 0 & -1 & 0 & 0 & 0
\end{bmatrix}\cdot
\eeqa}(Cf. with \eqref{eqn:Hbasis} and \eqref{eqn:oring}.)  In particular, $\hsh$ is an involution: $\hsh^2 = +1.$  Being a deformation of $g$ has the following additional advantage:

\begin{lemma}
\label{lemma:last}
Let $(M,g)$ be an oriented Riemannian 4-manifold and $h$ a split-signature deformation of $g$.  Then \emph{$\hsh$} is self-adjoint with respect to \emph{$\ipr{\,}{}$}.
\end{lemma}

\begin{proof}
Let $\{e_1,e_2,e_3,e_4\}$ be a $g$- and $h$-orthonormal frame, with $e_1,e_2$ $h$-timelike.  For $i < j$, observe that
$$
\ipr{\hsh(\ww{e_1}{e_2})}{\ww{e_i}{e_j}} = \left\{\begin{array}{ll}
0 & \text{if $\ww{e_i}{e_j} \neq \ww{e_3}{e_4}$,}\\
1 & \text{if $\ww{e_i}{e_j} = \ww{e_3}{e_4}$}.
\end{array}\right\} = \ipr{\ww{e_1}{e_2}}{\hsh(\ww{e_i}{e_j})}.
$$
The remaining cases are similar.
\end{proof}

(Note that $\co$ is not $\iph{\,}{}$-self-adjoint.) The  analogue of Definition \ref{def:*} is:

\begin{defn}[$\hsh$-Einstein metric]
\label{def:*h}
An oriented Riemannian 4-manifold $(M,g)$ is \emph{$\hsh$-Einstein} if the curvature operator of $g$ commutes with the Hodge star operator \emph{$\hsh$} of a split-signature metric $h$ that is a deformation of $g$.
\end{defn}

Although there is a direct analogue of Proposition \ref{prop:crit} in the split-signature case, by which one may prove an analogue of Theorem \ref{thm:n}, we forego doing so; rather,
using Lemma \ref{lemma:last}, we show that in the split-signature case there is a stronger notion of ``almost-Einstein" generalizing both $\Ric = \lambda g$ and \eqref{eqn:gsec_def}:

\begin{thm}[almost-Einstein metric, split-signature case]
\label{thm:sech}
Let $(M,g)$ be an oriented Riemannian 4-manifold and $h$ a split-signature metric on $M$ that is a deformation of $g$.  Then $(M,g)$ is \emph{$\hsh$}-Einstein if and only if
\beqa
\label{def:pEh}
\text{\emph{tr}}_{\scalebox{0.6}{h}}(\text{\emph{Rm}}_{\scalebox{0.6}{g}}) = f h,
\eeqa
where $\text{\emph{Rm}}_{\scalebox{0.6}{g}}$ is the Riemann curvature 4-tensor of $g$ and $f$ is a smooth function on $M$.  Equivalently, $(M,g)$ is \emph{$\hsh$-Einstein} if and only if
\beqa
\label{def:pEh2}
\text{\emph{$\gsec(\Ph) = \gsec(P)$}}
\eeqa
for all 2-planes $P$, where \emph{$\Ph \defeq \hsh P$}.
\end{thm} 

\begin{proof}
With respect to an oriented $g$- and $h$-orthonormal frame $\{e_1,e_2,e_3,e_4\}$, with $e_1,e_2$ $h$-timelike, the $\hsh$-Einstein condition $\hsh \circ \co = \co \circ \hsh$, together with \eqref{eqn:co_sum} and \eqref{eqn:Hhbasis}, yields the following identity:
{\footnotesize
$$
\hspace{-.2in}\begin{bmatrix}
R_{1234} & R_{1334} & R_{1434} & R_{3434} & R_{4234} & R_{2334}\\
-R_{1242} & -R_{1342} & -R_{1442} & -R_{3442} & -R_{4242} & -R_{2342}\\
-R_{1223} & -R_{1323} & -R_{1423} & -R_{3423} & -R_{4223} & -R_{2323}\\
R_{1212} & R_{1312} & R_{1412} & R_{3412} & R_{4212} & R_{2312}\\
-R_{1213} & -R_{1313} & -R_{1413} & -R_{3413} & -R_{4213} & -R_{2313}\\
-R_{1214} & -R_{1314} & -R_{1414} & -R_{3414} & -R_{4214} & -R_{2314}
\end{bmatrix}
$$}
{\footnotesize
\beqa\label{eqn:same} \hspace{1in}= \begin{bmatrix}
R_{3412} & -R_{4212} & -R_{2312} & R_{1212} & -R_{1312} & -R_{1412}\\
R_{3413} & -R_{4213} & -R_{2313} & R_{1213} & -R_{1313} & -R_{1413}\\
R_{3414} & -R_{4214} & -R_{2314} & R_{1214} & -R_{1314} & -R_{1414}\\
R_{3434} & -R_{4234} & -R_{2334} & R_{1234} & -R_{1334} & -R_{1434}\\
R_{3442} & -R_{4242} & -R_{2342} & R_{1242} & -R_{1342} & -R_{1442}\\
R_{3423} & -R_{4223} & -R_{2323} & R_{1223} & -R_{1323} & -R_{1423}
\end{bmatrix}\cdot
\eeqa}(Here $R_{ijkl} \defeq \text{Rm}_{\scalebox{0.6}{\emph{g}}}(e_i,e_j,e_k,e_l)$ as usual.) Setting $H \defeq \text{tr}_{\scalebox{0.6}{\emph{h}}}(\text{Rm}_{\scalebox{0.6}{\emph{g}}})$, observe that \eqref{eqn:same} is in turn equivalent to the off-diagonal terms of $H$ vanishing,
\beqa
\arraycolsep=1.4pt\def\arraystretch{1.5}
\left\{\begin{array}{lr}
H_{12} = R_{3123} + R_{4124} = 0,\nonumber\\
H_{13} = -R_{2132} + R_{4134} = 0,\nonumber\\
H_{14} = -R_{2142} + R_{3143} = 0,\nonumber\\
H_{23} = -R_{1231} + R_{4234} = 0,\nonumber\\
H_{24} = -R_{1241} + R_{3243} = 0,\nonumber\\
H_{34} = -R_{1341} - R_{2342} = 0,\nonumber
\end{array}\right.
\eeqa
together with all of its diagonal terms being equal up to the sign of $h_{ii}$:
\beqa
\arraycolsep=1.4pt\def\arraystretch{1.5}
\left\{\begin{array}{lr}
H_{22} = -R_{1221} + \underbrace{\,R_{3223}\,}_{\text{$R_{4114}$}}+\underbrace{\,R_{4224}\,}_{\text{$R_{3113}$}} = H_{11},\nonumber\\
H_{33} = -R_{1331}-R_{2332}+\underbrace{\,R_{4334}\,}_{\text{$R_{2112}$}} = -H_{11},\nonumber\\
H_{44} = -R_{1441}-R_{2442}+R_{3443} = -H_{11}.\nonumber
\end{array}\right.
\eeqa
Finally, since we know by Lemma \ref{lemma:last} that $\hsh$ is also self-adjoint with respect to $g$'s inner product $\ipr{\,}{}$ on $\Lambda^2$, we have, for  any  2-plane $P$,
$$
\gsec(\hsh P) = \ipr{\co(\hsh P)}{\hsh P} = \ipr{\hsh(\co P)}{\hsh P} = \ipr{\co P}{P} = \gsec(P).
$$
Conversely, $\gsec(\hsh P) = \gsec(P)$ implies that the operators $\hsh \circ\co \circ \hsh$ and $\co$ have equal $\ipr{\,}{}$-quadratic forms.  Since now \emph{both} satisfy the algebraic Bianchi identity, they must be equal (cf.~\eqref{eqn:cont}), hence $g$ is $\hsh$-Einstein.
\end{proof}

Before proceeding to the next case, let us pause to give a ``bird's-eye view" of the differences between the Lorentzian and split-signature cases:
\begin{enumerate}[leftmargin=*]
\item[i.] The split-signature case yields a more satisfactory notion of ``almost-Einstein" metric (Theorem \ref{thm:sech}) than the Lorentzian case (Theorem \ref{thm:secsec}).  This is because $\hsh$ is self-adjoint with respect to $\ipr{\,}{}$, but $\hsl$ is not.
\item[ii.] On the other hand, the Lorentzian case yields a more satisfactory relation to the critical points of $\gsec$ (Theorem \ref{thm:n}).  This is because complexifying $\Lambda^2$ via $\hsl^2=-1$ allowed us to directly relate certain critical points of $\gsec$ to the (complex) eigenvectors of $\co\colon \cx \lra \cx$ (recall \eqref{eqn:thorpecrit}), and thus to $g$'s Petrov Type.  No such complexification exists in the split-signature case.  Moreover, although in the split-signature case $\co \circ \hsh = \hsh \circ \co$ yields a $\ipr{\,}{}$-orthonormal basis of eigenvectors for both $\co$ \emph{and} $\hsh$, these needn't yield 2-planes as they do in the purely Riemannian setting, and which \eqref{eqn:thorpecrit} requires (whereas any sum of unit self- and anti-self-dual eigenvectors of $\hsr$ is necessarily decomposable, hence a 2-plane (see \cite{thorpe2}), this is not true with the eigenvectors of $\hsh$).
\end{enumerate}

\section{The case where $g$ is Lorentzian and $h$ is Riemannian}
\label{sec:L}
Let us now consider an oriented Lorentzian 4-manifold $(M,\gL)$, a unit timelike vector field $T$, and a \emph{Riemannian} deformation $h \defeq \gL + 2T^{\flat} \otimes T^{\flat}$, where $T^{\flat} \defeq \gL(T,\cdot)$.  Here, instead of $\gL$'s curvature operator $\col$, let us work instead with the endomorphism $\cooh\colon \Lambda^2\lra \Lambda^2$ defined by
\beqa
\label{eqn:lasth}
R_{ijkl} \defeq -\iph{\cooh(\ww{e_i}{e_j})}{\ww{e_k}{e_l}},
\eeqa
where $R_{ijkl}$ are the components of the Riemann curvature 4-tensor $\text{Rm}_{\scalebox{0.4}{\emph{L}}}$ of $\gL$; by construction, $\cooh$ is $\iph{\,}{}$-self-adjoint.  In terms of a $\gL$- and $h$-orthonormal frame $\{T\defeq e_1,e_2,e_3,e_4\}$, the difference between $\cooh$ and $\col$ is that
{\tiny
$$
\hspace{-.2in}\col = -\!\begin{bmatrix}
-R_{1212} & -R_{1312} & -R_{1412} & -R_{3412} & -R_{4212} & -R_{2312}\\
-R_{1213} & -R_{1313} & -R_{1413} & -R_{3413} & -R_{4213} & -R_{2313}\\
-R_{1214} & -R_{1314} & -R_{1414} & -R_{3414} & -R_{4214} & -R_{2314}\\
R_{1234} & R_{1334} & R_{1434} & R_{3434} & R_{4234} & R_{2334}\\
R_{1242} & R_{1342} & R_{1442} & R_{3442} & R_{4242} & R_{2342}\\
R_{1223} & R_{1323} & R_{1423} & R_{3423} & R_{4223} & R_{2323}
\end{bmatrix},
$$}
{\tiny
\beqa\label{eqn:RR}\hspace{1in}\text{whereas}~\cooh = -\!\begin{bmatrix}
R_{1212} & R_{1312} & R_{1412} & R_{3412} & R_{4212} & R_{2312}\\
R_{1213} & R_{1313} & R_{1413} & R_{3413} & R_{4213} & R_{2313}\\
R_{1214} & R_{1314} & R_{1414} & R_{3414} & R_{4214} & R_{2314}\\
R_{1234} & R_{1334} & R_{1434} & R_{3434} & R_{4234} & R_{2334}\\
R_{1242} & R_{1342} & R_{1442} & R_{3442} & R_{4242} & R_{2342}\\
R_{1223} & R_{1323} & R_{1423} & R_{3423} & R_{4223} & R_{2323}
\end{bmatrix}\cdot
\eeqa}

(In particular, $\cooh$ is symmetric.) The corresponding notion of ``almost-Einstein" metric in this case is the ``mirror image" of Definition \ref{def:*}:
\begin{defn}[$\hsh$-Einstein metric]
\label{def:*h2}
Let \emph{$(M,\gL)$} be an oriented Lorentzian 4-manifold, $T$ a unit timelike vector field on $M$, and $h$ the Riemannian metric \emph{$h \defeq \gL+2 T^\flat \otimes T^\flat$}. Then \emph{$(M,\gL)$} is \emph{$\hsh$-Einstein} if the operator \emph{$\cooh\colon \Lambda^2\lra \Lambda^2$} defined by \eqref{eqn:lasth} commutes with the Hodge star operator \emph{$\hsh$} of $h$.
\end{defn}
This definition generalizes the notion of gravitational Petrov Type:

\begin{thm}[Normal Form for Lorentzian $\hsh$-Einstein metrics]
\label{thm:last}
An oriented Lorentzian 4-manifold \emph{$(M,\gL)$} is \emph{$\hsh$}-Einstein with respect to the Riemannian metric \emph{$h \defeq \gL+2 T^\flat \otimes T^\flat$} if and only if
\beqa
\label{eqn:RRR}
\text{\emph{tr}}_{\scalebox{0.6}{h}}(\text{\emph{Rm}}_{\scalebox{0.4}{L}}) = f h,
\eeqa
where $\text{\emph{Rm}}_{\scalebox{0.4}{L}}$ is the Riemann curvature 4-tensor of \emph{$\gL$} and $f$ is a smooth function on $M$.  Any \emph{$\hsl$}-Einstein manifold \emph{$(M,\gL)$} has a unique normal form that is completely determined by the critical points and values of the quadratic form \emph{$P \mapsto \iph{\cooh P}{P}$, defined for all 2-planes $P$}.
\end{thm}

\begin{proof}
\eqref{eqn:RRR} follows from \eqref{eqn:RR} and $\hsh = {\tiny \begin{bmatrix}
O & I\\
I & O
\end{bmatrix}}$, as \eqref{def:pEh} did from \eqref{eqn:same} in Theorem \ref{thm:sech}.  Regarding $(M,\gL)$'s normal form, this follows by the same proof as in \cite[Theorems~2.1 \& 2.2]{thorpe2}.  Indeed, to utilize those proofs, three conditions must hold: i.) The critical points of $P \mapsto \iph{\cooh P}{P}$ would need to be characterized by $\cooh P = aP+b(\hsh\hspace{.01in}P)$ (one can show this as in Proposition \ref{prop:crit}); ii.)~$\cooh$ and $\hsh$ would need to be self-adjoint with respect to the (positive-definite) inner product $\iph{\,}{}$, which they are; iii.)~any 2-plane would need to be expressible as $\frac{1}{\sqrt{2}}(\alpha + \beta)$, where $\alpha$ and $\beta$ are, respectively, unit self-dual and anti-self-dual eigenvectors of $\hsh$.  This is also the case.
\end{proof}

``Normal form" means that at each $p\in M$, there is an $h$-orthonormal frame $\{e_1,e_2,e_3,e_4\}$ such that each 2-plane $\ww{e_i}{e_j}$ in the corresponding basis \eqref{eqn:Hbasis0} for $\Lambda^2$, which we may write in the form $\{P_1,P_2,P_3,{P_1}^{\!\perp_{\scalebox{0.5}{\emph{h}}}},{P_2}^{\!\perp_{\scalebox{0.5}{\emph{h}}}},{P_3}^{\!\perp_{\scalebox{0.5}{\emph{h}}}}\}$, is a critical point of $P \mapsto \iph{\cooh P}{P}$, hence satisfies $\cooh P_i = \lambda_i P_i + \mu_i{P_i}^{\perp_{\scalebox{0.5}{\emph{h}}}}$. Thus
{\tiny
$$
\cooh = -\!\begin{bmatrix}
\lambda_1 & 0 & 0 & \mu_1 & 0 & 0\\
0 & \lambda_2 & 0 & 0 & \mu_2 & 0\\
0 & 0 & \lambda_3 & 0 & 0 & \mu_3\\
\mu_1 & 0 & 0 & \lambda_1 & 0 & 0\\
0 & \mu_2 & 0 & 0 & \lambda_2 & 0\\
0 & 0 & \mu_3 & 0 & 0 & \lambda_3
\end{bmatrix}\cdot
$$}(If we could arrange it so that $T \defeq e_1$, then $\col$ would also inherit this normal form, via \eqref{eqn:RR}.) In  general, this normal form exists for $\hat{S}^{\scalebox{0.5}{\emph{h}}}_{\scalebox{0.4}{\emph{L}}} \defeq \frac{1}{2}(\cooh + \hsh \circ \cooh \circ \hsh)$ (which is $\iph{\,}{}$-self-adjoint), as it also would if $g$ and $h$ were both Riemannian (if they had orthonormal frames in common, then \eqref{eqn:RRR} would hold, too).
\begingroup
\bibliographystyle{alpha}
\bibliography{Hodge_star-gL}

\newcommand{\etalchar}[1]{$^{#1}$}
\begin{thebibliography}{SKM{\etalchar{+}}09}

\bibitem[AG97]{apostolov97}
V.~Apostolov and P.~Gauduchon.
\newblock The {R}iemannian {G}oldberg-{S}achs theorem.
\newblock {\em International Journal of Mathematics}, 8(04):421--439, 1997.

\bibitem[Ber61]{berger}
Marcel Berger.
\newblock Sur quelques vari{\'e}t{\'e}s d'{E}instein compactes.
\newblock {\em Annali di Matematica Pura ed Applicata}, 53(1):89--95, 1961.

\bibitem[Bes07]{besse}
Arthur~L Besse.
\newblock {\em Einstein manifolds}.
\newblock Springer Science \& Business Media, 2007.

\bibitem[CGT23]{gursky}
Xiaodong Cao, Matthew~J. Gursky, and Hung Tran.
\newblock Curvature of the second kind and a conjecture of {N}ishikawa.
\newblock {\em Commentarii Mathematici Helvetici}, 98(1):195--216, 2023.

\bibitem[CP80]{CP}
Michel Cahen and Monique Parker.
\newblock {\em Pseudo-{R}iemannian symmetric spaces}, volume 229.
\newblock American Mathematical Society, 1980.

\bibitem[CT16]{cao}
Xiaodong Cao and Hung Tran.
\newblock The {W}eyl tensor of gradient {R}icci solitons.
\newblock {\em Geometry \& Topology}, 20(1):389--436, 2016.

\bibitem[Der00]{derdzinski}
Andrzej Derdzinski.
\newblock Einstein metrics in dimension four.
\newblock {\em Handbook of Differential Geometry}, 1:419--707, 2000.

\bibitem[DeT81]{deturck}
Dennis~M. DeTurck.
\newblock Existence of metrics with prescribed {R}icci curvature: local theory.
\newblock {\em Inventiones Mathematicae}, 65(2):179--207, 1981.

\bibitem[Hit74]{HT-hitchin}
Nigel Hitchin.
\newblock Compact four-dimensional {E}instein manifolds.
\newblock {\em Journal of Differential Geometry}, 9(3):435--441, 1974.

\bibitem[Joh80a]{johnson}
David~L. Johnson.
\newblock A curvature normal form for 4-dimensional {K}{\"a}hler manifolds.
\newblock {\em Proceedings of the American Mathematical Society},
  79(3):462--464, 1980.

\bibitem[Joh80b]{johnson2}
David~L. Johnson.
\newblock Sectional curvature and curvature normal forms.
\newblock {\em Michigan Math. J}, 27(1):980, 1980.

\bibitem[Li23]{li}
Xiaolong Li.
\newblock Manifolds with nonnegative curvature operator of the second kind.
\newblock {\em Communications in Contemporary Mathematics}, to appear, 2023.

\bibitem[Ole14]{olea}
Benjam{\'\i}n Olea.
\newblock Canonical variation of a {L}orentzian metric.
\newblock {\em Journal of Mathematical Analysis and Applications},
  419(1):156--171, 2014.

\bibitem[O'N83]{o1983}
Barrett O'Neill.
\newblock {\em Semi--{R}iemannian {G}eometry with {A}pplications to
  {R}elativity}, volume 103.
\newblock Academic press, 1983.

\bibitem[O'N95]{o1995}
Barrett O'Neill.
\newblock The {G}eometry of {K}err {B}lack {H}oles.
\newblock {\em Wellesley, Mass.: AK Peters}, 1, 1995.

\bibitem[Pet69]{petrov}
A.~Z. Petrov.
\newblock {\em Einstein spaces}.
\newblock Pergamon Press, 1969.

\bibitem[Pet16]{PP}
Peter Petersen.
\newblock {\em Riemannian geometry}, volume 171.
\newblock Springer, 2016.

\bibitem[SKM{\etalchar{+}}09]{stephani2009}
Hans Stephani, Dietrich Kramer, Malcolm MacCallum, Cornelius Hoenselaers, and
  Eduard Herlt.
\newblock {\em Exact solutions of {E}instein's field equations}.
\newblock Cambridge University Press, 2009.

\bibitem[ST69]{thorpe2}
Isadore~M. Singer and John~A. Thorpe.
\newblock {\em \emph{The curvature of 4-dimensional {E}instein spaces. In}
  Global Analysis\emph{:} Papers in Honor of K. Kodaira, \emph{pages
  355--365}}.
\newblock University of Tokyo Press, 1969.

\bibitem[Tho69a]{thorpe}
John~A. Thorpe.
\newblock Curvature and the {P}etrov canonical forms.
\newblock {\em Journal of Mathematical Physics}, 10(1):1--7, 1969.

\bibitem[Tho69b]{HT-thorpe}
John~A. Thorpe.
\newblock Some remarks on the {G}auss-{B}onnet integral.
\newblock {\em Journal of Mathematics and Mechanics}, 18(8):779--786, 1969.

\bibitem[Tho71]{thorpe4}
John~A. Thorpe.
\newblock The zeros of nonnegative curvature operators.
\newblock {\em Journal of Differential Geometry}, 5(1-2):113--125, 1971.

\bibitem[Tho72]{thorpe3}
John~A. Thorpe.
\newblock On the curvature tensor of a positively curved 4-manifold.
\newblock In {\em Proceedings of the Thirteenth Biennial Seminar of the
  Canadian Mathematical Congress (Dalhousie Univ., Halifax, NS, 1971)},
  volume~2, pages 156--159, 1972.

\bibitem[Zol79]{zoltek}
Stanley~M. Zoltek.
\newblock Nonnegative curvature operators: some nontrivial examples.
\newblock {\em Journal of Differential Geometry}, 14(2):303--315, 1979.

\bibitem[Zol80]{zoltek2}
Stanley~M. Zoltek.
\newblock A normal form for a special class of curvature operators.
\newblock {\em Proceedings of the American Mathematical Society},
  79(4):614--618, 1980.

\end{thebibliography}
\endgroup
\end{document}